\documentclass[12pt,a4paper,reqno]{amsart}
\usepackage{amsmath,amsthm,amssymb,mathtools,enumitem,graphicx}
\usepackage{a4wide}

\usepackage[pagebackref,colorlinks=true,citecolor=blue,linkcolor=red,urlcolor=blue]{hyperref}

\title[On \MakeLowercase{$p$}-groups related to $G_2($\MakeLowercase{$p$}$)$]{On $p$-groups with automorphism groups\\ related to the Chevalley group $G_2(p)$}

\author[J. Bamberg]{JOHN BAMBERG}
\address{\parbox{\linewidth}{ %
J. BAMBERG, Centre for the Mathematics of Symmetry and Computation,\\
The University of Western Australia, 35 Stirling Highway,\\
Crawley, WA 6009, Australia}}
\email{john.bamberg@uwa.edu.au}
\author[S. D. Freedman]{SAUL D. FREEDMAN}
\address{\parbox{\linewidth}{ %
S. D. FREEDMAN, Centre for the Mathematics of Symmetry and Computation,\\
The University of Western Australia, 35 Stirling Highway,\\
Crawley, WA 6009, Australia\\
\textit{Current address}: School of Mathematics and Statistics,\\
University of St Andrews, St Andrews, KY16 9SS, UK}}
\email{sdf8@st-andrews.ac.uk}
\author[L. Morgan]{LUKE MORGAN}
\address{\parbox{\linewidth}{ %
L. MORGAN, Centre for the Mathematics of Symmetry and Computation,\\
The University of Western Australia, 35 Stirling Highway,\\
Crawley, WA 6009, Australia\\
\textit{Current address}: UP FAMNIT,\\
University of Primorska, Glagolja\v{s}ka 8, 6000 Koper, Slovenia\\
\textit{Also affiliated with}: UP IAM,\\
University of Primorska, Muzejski trg 2, 6000 Koper, Slovenia}}
\email{luke.morgan@famnit.upr.si}

\date{\today}
\thanks{\textit{\keywordsname}. $p$-group, exterior square, $G_2(q)$\\
\indent The first author acknowledges the support of the Australian Research Council Future Fellowship FT120100036. The second author was supported by a Hackett Foundation Alumni Honours Scholarship, a Hackett Postgraduate Research Scholarship, and an Australian Government Research Training Program Scholarship at The University of Western Australia. The third author was supported by the Australian Research Council grant DE160100081.}

\newtheorem{thm}{Theorem}[section]
\newtheorem*{thm*}{Theorem}
\newtheorem*{conj*}{Conjecture}

\newtheorem{lem}[thm]{Lemma}
\newtheorem{prop}[thm]{Proposition}

\theoremstyle{remark}

\theoremstyle{definition}
\newtheorem{defn}[thm]{Definition}
\newcounter{claim}[thm]

\numberwithin{equation}{section}

\renewcommand{\le}{\leqslant}
\renewcommand{\ge}{\geqslant}
\newcommand{\scO}{\mathcal{O}}

\begin{document}

\begin{abstract}
Let $p$ be an odd prime. We construct a $p$-group $P$ of nilpotency class two, rank seven and exponent $p$, such that $\mathrm{Aut}(P)$ induces $N_{\mathrm{GL}(7,p)}(G_2(p)) = Z(\mathrm{GL}(7,p)) G_2(p)$ on the Frattini quotient $P/\Phi(P)$. The constructed group $P$ is the smallest $p$-group with these properties, having order $p^{14}$, and when $p = 3$, our construction gives two nonisomorphic $p$-groups. To show that $P$ satisfies the specified properties, we study the action of $G_2(q)$ on the octonion algebra over $\mathbb{F}_q$, for each power $q$ of $p$, and explore the reducibility of the exterior square of each irreducible seven-dimensional $\mathbb{F}_q[G_2(q)]$-module.
\end{abstract}

\maketitle

\section{Introduction}
\label{sec:intro}

\noindent The topic of this paper is the problem of representing groups on $p$-groups. The traditional version involves linear representations on vector spaces, that is, on elementary abelian $p$-groups; here we study a nonabelian analogue. As in the traditional case, it is natural to consider simple groups, since this may lead to an understanding of the problem for finite groups in general. In this paper, we begin a programme to study the nonabelian analogue for the exceptional groups of Lie type; the classical groups of Lie type were studied in \cite{bamberg}. The groups of type $G_2$ are the nontwisted exceptional groups of lowest Lie rank, and so we focus first on those. These groups are also interesting from a purely algebraic point of view, being automorphism groups of octonion algebras. To describe our goal more precisely, we require some definitions.

Let $P$ be a finite $p$-group. The \emph{Frattini subgroup} $\Phi(P)$ is a characteristic subgroup of $P$, and therefore there is a natural homomorphism $\theta$ from $\mathrm{Aut}(P)$ to $\mathrm{Aut}(P/\Phi(P))$, the automorphism group of the \emph{Frattini quotient} $P/\Phi(P)$ of $P$. We also know from Burnside's basis theorem that $P/\Phi(P)$ can be identified with the vector space $\mathbb{F}_p^d$, where $d$ is the \emph{rank} of $P$, that is, the minimum size of a generating set for $P$. Therefore, the group $A(P)$ \emph{induced} by $\mathrm{Aut}(P)$ on $P/\Phi(P)$, that is, the image of $\theta$, can be identified with a subgroup of the general linear group $\mathrm{GL}(d,p)$. In fact, this identification preserves the linear action of the group on $P/\Phi(P) \cong \mathbb{F}_p^d$.

Bryant and Kov\'acs \cite{bryant} proved that if $H$ is any subgroup of $\mathrm{GL}(d,p)$, with $d > 1$, then there exists a $p$-group $P$ of rank $d$ such that $A(P) = H$. From the proof of their result, we see that both the class and exponent of the group $P$ are roughly $|\mathrm{GL}(d,p)|$ (depending on both $d$ and $p$), and the order of $P$ is not explicit, but must be enormous. Bamberg \textit{et al.}~\cite{bamberg} were inspired to investigate the necessity of the large exponent and class of the groups constructed in \cite{bryant}. They showed that if $H$ satisfies a few extra conditions, then there exists a $p$-group $P$ with $A(P) = H$, such that $P$ has rank $d$; nilpotency class two, three or four; exponent $p$; and order at most $p^{d^4/2}$. Here, $H$ is a maximal subgroup of $\mathrm{GL}(d,p)$ that does not contain $\mathrm{SL}(d,p)$, where $d > 1$ and $p > 3$. The proof of the existence of $P$ in this case also requires that $H$ lies in a particular subset of the \emph{Aschbacher classes} of $\mathrm{GL}(d,p)$, where the Aschbacher classes of a classical group describe its subgroups, via Aschbacher's theorem \cite{aschbacherthm}. Bamberg \textit{et al.}~also showed that there are cases with $p = 3$ where there exists a corresponding $p$-group of nilpotency class two.

In odd characteristic, the \emph{natural representation} of a finite Chevalley group of type $G_2$ is on a seven-dimensional vector space, and hence we can consider $G_2(p)$ (for $p$ odd) as a subgroup of $\mathrm{GL}(7,p)$. We determine below a relatively small nonabelian counterpart to this representation. The class, exponent and order of the $p$-group $P$ shown to exist in the theorem below are drastically lower than the bounds that can be deduced from \cite{bryant}. Moreover, $P$ is the smallest $p$-group of the given nilpotency class, rank and exponent whose associated induced group is $Z(\mathrm{GL}(7,p)) G_2(p)$.

\begin{thm}
\label{thm:main}
Let $p$ be an odd prime. There exists a $p$-group $P$ of nilpotency class two, rank seven, exponent $p$ and order $p^{14}$, such that $A(P)= Z(\mathrm{GL}(7,p)) G_2(p)$. If $p = 3$, then there exist two such $p$-groups that are nonisomorphic.
\end{thm}

As we show in this paper, $Z(\mathrm{GL}(7,p)) G_2(p)$ is the normaliser of $G_2(p)$ in $\mathrm{GL}(7,p)$. This normaliser is not a maximal subgroup of the general linear group, and hence it does not satisfy the conditions required by Bamberg \textit{et al.} The methodology required to prove Theorem \ref{thm:main} is therefore different from that used in \cite{bamberg}, and depends upon realising $G_2(p)$ as the automorphism group of the octonion algebra over $\mathbb{F}_p$. This elementary approach highlights the connection between the groups of type $G_2$ and this algebra, and emphasises the uniqueness of these groups among the other groups of Lie type. In a sequel paper (`on $p$-groups with automorphism groups related to the exceptional Chevalley groups', submitted for publication), the second author proves similar results for each family of nontwisted exceptional groups, using a more Lie-theoretic approach.

Our method of proof for Theorem \ref{thm:main} is to appeal to a description of groups of class two and exponent $p$, as in \cite[\S2--3]{eick}, \cite[Ch.~9.4]{holt} and \cite[\S2]{obrien} (see also \cite[\S2]{bamberg}). Such a $p$-group $P$ is a quotient of the \emph{universal $p$-group} of nilpotency class two, rank $d$ and exponent $p$, that is, the largest finite group with these properties. With $V:= \mathbb{F}_p^d$, there exists a proper subspace $U$ of $A^2 V$, the exterior square of $V$, such that $P= P_U$, where
\begin{equation}
\label{setofP} P_U := V \times A^2V/U,
\end{equation}
equipped with the multiplication defined by
\begin{equation*} (v,w+U)(v',w'+U) := (v+v',w+w'+v\wedge v' +U). \end{equation*}
Moreover, the automorphisms induced on the Frattini quotient of $P$ are easy to understand in that $A(P) = N_{\mathrm{GL}(d,p)}(U)$. Conversely, for each proper subspace $U$ of $A^2 V$, the group $P_U$ of class two and exponent $p$ defined in the way above enjoys the property $A(P_U)= N_{\mathrm{GL}(d,p)}(U)$.
Thus our strategy to exhibit the group $P$ in Theorem~\ref{thm:main} is to consider the reducibility of the exterior square $A^2V$ of a seven-dimensional module $V$ for $G_2(p)$.

It must be remarked that the stabiliser in $\mathrm{GL}(d,p)$ of any subspace of $A^2V$ contains the centre of the general linear group, which acts via scalar multiplication. Hence the group $Z(\mathrm{GL}(7,p)) G_2(p)$ from Theorem \ref{thm:main} is the smallest group containing $G_2(p)$ that we can induce on the Frattini quotient of a $p$-group of nilpotency class two, rank seven and exponent $p$.

The paper is organised as follows. In \S\ref{sec:reducibility}, we explore the reducibility of the exterior square of each irreducible seven-dimensional $G_2(p)$-module in order to find a suitable subspace $U$ to substitute in \eqref{setofP}. In particular, we determine the maximum dimension of such a subspace in order to determine the minimum order of the $p$-group $P_U$. These results regarding reducibility are interesting in their own right, and may be useful in other applications, and hence we explore the more general case of the Chevalley group $G_2(q)$, where $q$ is a power of an odd prime. Some of our results are even more general, and apply to the larger family of groups $G_2(\mathbb{F})$, where $\mathbb{F}$ is any field of characteristic not equal to two. In \S\ref{sec:overgroups}, we explore part of the overgroup structure of $G_2(q)$ in $\mathrm{GL}(7,q)$. We then use the results of \S\ref{sec:reducibility}--\ref{sec:overgroups} 
 to prove Theorem \ref{thm:main} in \S\ref{sec:mainproof}. Finally, we conclude in \S\ref{sec:concl} with suggestions for future research.

\section{Reducibility of the exterior square of an irreducible seven-dimensional $G_2(q)$-module}
\label{sec:reducibility}

In this section, we explore the reducibility of the exterior square of each irreducible seven-dimensional $G_2(q)$-module, with $q$ an odd prime power. We note that this reducibility was investigated in \cite[Ch.~9.3.2]{schroder}. Here, we consider a totally different approach that also yields results in the more general case of the group $G_2(\mathbb{F})$, where $\mathbb{F}$ is any field of characteristic not equal to two.

\begin{lem}[{\cite[Ch.~4.3.2]{wilson}}]
\label{lem:octonions}
Let $\mathbb{O}$ be the eight-dimensional octonion algebra over $\mathbb{F}$, with basis $\{1,i_0,i_1,\ldots,i_6\}$ and multiplication defined in terms of basis vectors as:
\begin{enumerate}[label={(\roman*)},font=\upshape]
\item $1e = e1 = e$ for each $e \in \{1, i_0, i_1, \ldots, i_6\}$;
\item $i_t^2 = -1$ for each $t \in \mathbb{Z}/7\mathbb{Z}$; and
\item $i_t i_{t+1} = i_{t+3}$, $i_{t+1} i_{t+3} = i_{t}$, and $i_{t+3} i_t = i_{t+1}$ for each $t \in \mathbb{Z}/7\mathbb{Z}$, with subscripts interpreted modulo seven.
\end{enumerate}
Then $G_2(\mathbb{F}) := \mathrm{Aut}(\mathbb{O})$, with $G_2(\mathbb{F}) = G_2(q)$ if $\mathbb{F}$ is the finite field of order $q$.
\end{lem}

This octonion algebra is an eight-dimensional analogue of the two-dimensional algebra of the complex numbers over the reals and the four-dimensional quaternion algebra. It follows from the definition of multiplication in $\mathbb{O}$ that $i_t i_s = -i_s i_t$ for all distinct ${s, t \in \mathbb{Z}/7\mathbb{Z}}$. Furthermore, $1$ is the unique multiplicative identity of $\mathbb{O}$. The elements of $\mathbb{O}$ are called \emph{octonions}, and the \emph{real octonions} (respectively, the \emph{imaginary octonions}) are the elements of $\langle 1 \rangle$ (respectively, of $\scO:=\langle i_0,i_1,\ldots,i_6\rangle$). If $x = b1 + \sum_{t=0}^6 a_t i_t$ is an octonion, with $b, a_0, a_1, \ldots, a_6 \in \mathbb{F}$, then the \emph{real part} of $x$ is $\mathrm{Re}(x) := b1$, its \emph{imaginary part} is $\mathrm{Im}(x):=\sum_{t=0}^6 a_t i_t \in \scO$, and its \emph{conjugate} is $\overline{x} := b1 - \sum_{t=0}^6 a_t i_t$. The function $\mathrm{Re}$ (respectively, $\mathrm{Im}$) is the natural $\mathbb{F}$-linear projection map from $\mathbb{O} = \langle 1 \rangle \oplus \scO$ to $\langle 1 \rangle$ (respectively, to $\scO$).

The following result and its proof are given by Wilson \cite[pages 119 and 121]{wilson}.

\begin{prop}
\label{prop:g2imact}
The group $G_2(\mathbb{F})$ stabilises $\scO$ (as a subspace of the algebra $\mathbb{O}$).
\end{prop}

\begin{proof}
If we identify $\langle 1 \rangle$ with $\mathbb{F}$, then we can define a bilinear form $\beta:\mathbb{O} \times \mathbb{O} \to \mathbb{F}$ by $(x,y)\beta := \mathrm{Re}(x\overline y)$, and $G_2(\mathbb{F})$ preserves this form. Since $G_2(\mathbb{F})$ is the automorphism group of $\mathbb{O}$, it fixes $1 \in \mathbb{O}$, and hence it stabilises $\langle 1 \rangle$. Therefore, $G_2(\mathbb{F})$ stabilises the orthogonal complement ${\langle 1 \rangle}^\perp$ of $\langle 1 \rangle$ with respect to $\beta$, which is $\scO$.
\end{proof}

\begin{lem}
\label{lem:ftildecoll}
Let $f:\scO^2 \to \scO$ be the map defined by $(x,y)f := \mathrm{Im}(xy)$ for $x,y \in \scO$. Then:
\begin{enumerate}[label={(\roman*)}, font=\upshape]
\item There exists a unique linear map ${\tilde f: A^2\scO \to \scO}$ defined by $(x \wedge y)\tilde f := (x,y)f$ for all $x, y \in \scO$. \label{prop:ftilde}
\item The map $\tilde f$ is an $\mathbb{F}[G_2(\mathbb{F})]$-homomorphism. \label{prop:Fwedgehom}
\item The group $G_2(\mathbb{F})$ stabilises the subspace $\ker \tilde f$ of $A^2\scO$ (as a vector space), which has dimension $14$. \label{lem:g214dim}
\end{enumerate}
\end{lem}

\begin{proof}
\begin{NoHyper}\ref{prop:ftilde}\end{NoHyper} It is clear that $f$ is bilinear, and it follows from the octonion multiplication rules that the square of an imaginary octonion is real, which means that $f$ is alternating. Thus by the universal property of the exterior square, there exists a unique linear map ${\tilde f: A^2\scO \to \scO}$ such that $(x \wedge y)\tilde f = (x,y)f$ for all $x, y \in \scO$.

\begin{NoHyper}\ref{prop:Fwedgehom}\end{NoHyper} For distinct basis vectors $i_s$ and $i_t$, we have $(i_s \wedge i_t)\tilde f = i_s i_t$, since this product is imaginary. As $G_2(\mathbb{F})$ acts linearly on $\scO$ by Proposition \ref{prop:g2imact}, the image $(i_s i_t)^g = i_s^g i_t^g$ is also imaginary for each $g \in G_2(\mathbb{F})$. The group $G_2(\mathbb{F})$ also acts linearly on $A^2\scO$, with $(x \wedge y)^g := x^g \wedge y^g$ for all $x,y \in \scO$ and $g \in G_2(\mathbb{F})$. Therefore, $$((i_s \wedge i_t)\tilde f)^g = (i_s i_t)^g = i_s^g i_t^g = (i_s^g \wedge i_t^g)\tilde f = ((i_s \wedge i_t)^g)\tilde f.$$ Since $A^2\scO$ is spanned by the wedge products $i_s \wedge i_t$, it follows that the linear map $\tilde f$ is an $\mathbb{F}[G_2(\mathbb{F})]$-homomorphism.

\begin{NoHyper}\ref{lem:g214dim}\end{NoHyper} As $\tilde f$ is an $\mathbb{F}[G]$-homomorphism by part \ref{prop:Fwedgehom}, $G_2(\mathbb{F})$ stabilises $\ker \tilde f$. Moreover, it follows from the definition of multiplication in $\mathbb{O}$ that, for each $t \in \mathbb{Z}/7\mathbb{Z}$, we have $(i_{t+1} \wedge i_{t+3})\tilde f = \mathrm{Im}(i_t) = i_t$. The image of $\tilde f$ is therefore equal to $\scO$, with dimension seven. We also have $\dim(A^2\scO) = \binom{\dim(\scO)}{2} = 21$, and hence $\dim(\ker \tilde f) = 21-7=14$.
\end{proof}

We have shown that when $\mathrm{char}(\mathbb{F}) \ne 2$, the group $G_2(\mathbb{F})$ stabilises a $14$-dimensional subspace of $A^2\scO$ corresponding to the alternating bilinear map given by $(x,y)f = \mathrm{Im}(xy)$ for $x,y \in \scO$. Gow \cite{gow} proves a stronger result when $\mathbb{F}$ is such that $\mathbb{O}$ is a division algebra, namely, that $G_2(\mathbb{F})$ stabilises a $14$-dimensional subspace of $A^2\scO$ corresponding to each alternating bilinear map from $\scO^2$ to $\scO$. However, our proof of Lemma \ref{lem:ftildecoll} holds even when $\mathbb{O}$ is a split algebra.

\begin{defn}[{\cite[Ch.~1.8.2]{BHRD}}]
\label{def:quaseq}
Let $G$ be a group, and let $W_1$ and $W_2$ be $\mathbb{F}[G]$-modules. Then $W_1$ and $W_2$ are \emph{quasiequivalent} if $W_1 \cong W_2^\alpha$ for some $\alpha \in \mathrm{Aut}(G)$, where $W_2^\alpha$ is the $\mathbb{F}[G]$-module obtained by \emph{twisting} $W_2$ by $\alpha$. Specifically, if $W_2$ affords the representation $\rho$, then $W_2^{\alpha}$ affords the representation $\rho_\alpha$, where $(g)\rho_\alpha := (g^\alpha)\rho$ for each $g \in G$.
\end{defn}

We now focus on the case where $\mathbb{F}$ is the finite field $\mathbb{F}_q$, with $q$ a power of an odd prime $p$, so that $G_2(\mathbb{F}) = G_2(q)$.

\begin{lem}
\label{lem:g2irrmod}
Up to isomorphism and twisting by field automorphisms, there are three distinct irreducible $\mathbb{F}_q[G_2(q)]$-modules of dimension at most $21$. These $G_2(q)$-modules are all absolutely irreducible, and have respective dimensions one, seven and seven if $p = 3$, and one, seven and $14$ otherwise. Furthermore, all irreducible $\mathbb{F}_q[G_2(q)]$-modules of fixed dimension $d \le 21$ are quasiequivalent, and the images of the afforded $\mathbb{F}_q$-representations are all conjugate in $\mathrm{GL}(d,q)$.
\end{lem}

\begin{proof}
Let $G$ be the linear algebraic group associated with $G_2(q)$, and let $K$ be the algebraic closure of the field $\mathbb{F}_q$. L\"ubeck \cite[Appendix A.49]{lubeck} shows that, up to isomorphism and twisting by field automorphisms, there are three distinct irreducible $K[G]$-modules of dimension at most $21$, with the same dimensions as the required $\mathbb{F}_q[G_2(q)]$-modules. As $G_2(q)$ is simply-connected, we have $G_2(q) < G$, and a theorem of Steinberg \cite[page 17]{humphreysmod} implies that the irreducible $K[G_2(q)]$-modules of dimension at most 21 are the restrictions to $G_2(q)$ of the aforementioned $K[G]$-modules. Another theorem of Steinberg \cite[page 42]{humphreysmod} implies that $\mathbb{F}_q$ is a splitting field for $G_2(q)$. Therefore, the irreducible $\mathbb{F}_q[G_2(q)]$-modules of dimension at most $21$ are all absolutely irreducible, and they can be identified with the $K[G_2(q)]$-modules by extending the scalars \cite[Corollary 9.8]{ichar}. Finally, when $p = 3$, the two distinct irreducible $\mathbb{F}_q[G_2(q)]$-modules of dimension seven are equivalent up to twisting by the exceptional graph automorphism of $G_2(q)$ \cite[pages 188--189]{humphreysmod}. The quasiequivalence result therefore follows from Definition \ref{def:quaseq}. Finally, since $G_2(q)$ is simple, each irreducible $G_2(q)$-module is either faithful or trivial. Hence the conjugacy result follows from the fact that faithful modules are quasiequivalent if and only if the afforded representations have conjugate images \cite[Lemma 1.8.6]{BHRD}.
\end{proof}

When $q = p$, the modules of Lemma \ref{lem:g2irrmod} are those distinct up to isomorphism, as $G_2(p)$ has no nontrivial field automorphisms.

\begin{prop}
\label{prop:g2octirred}
Let $V$ be a faithful seven-dimensional $\mathbb{F}_q[G_2(q)]$-module. Then $V$ is irreducible. In particular, $\scO$ is an irreducible $\mathbb{F}_q[G_2(q)]$-module.
\end{prop}

\begin{proof}
Suppose that $V$ is reducible. Then we have from Lemma \ref{lem:g2irrmod} that each $\mathbb{F}_q[G_2(q)]$-composition factor of $V$ is the unique one-dimensional $\mathbb{F}_q[G_2(q)]$-module, that is, the trivial irreducible module. Since $G_2(q)$ is perfect, the trivial module has no nontrivial self-extensions. This implies that $G_2(q)$ acts trivially on $V$, contradicting faithfulness.
\end{proof}

For a corresponding result when $\mathbb{F}$ is an infinite field, see \cite[Proposition 4]{jacobsonexcep}.

\begin{thm}
\label{thm:g2maxsubmod}
Let $V$ be an irreducible seven-dimensional $\mathbb{F}_q[G_2(q)]$-module. Then the $\mathbb{F}_q[G_2(q)]$-module $A^2V$ contains a submodule of dimension $14$, but no proper submodule of higher dimension.
\end{thm}

\begin{proof}
Let $U$ be a proper submodule of $A^2V$ with the highest possible dimension. Then $U$ is a maximal submodule of $A^2V$, and thus the correspondence theorem implies that the quotient $W:=A^2V/U$ is an irreducible $\mathbb{F}_q[G_2(q)]$-module. Suppose that $\dim(W)=1$, so that $W \cong \mathbb{F}_q$. Then $W$ is trivial, as $G_2(q)$ is nonabelian and simple. Additionally, the map $\beta: V \times V \to W$ defined by $(x,y)\beta:= x \wedge y + U$ for $x, y \in V$ is an alternating bilinear form on $V$. Since $U$ is $G_2(q)$-invariant and $W$ is trivial, if $g \in G_2(q)$, then $$(x^g,y^g)\beta = x^g \wedge y^g + U = (x \wedge y)^g + U = (x \wedge y + U)^g = x \wedge y+U = (x,y)\beta.$$ Thus $G_2(q)$ preserves $\beta$, which is clearly nonzero as $U$ is a proper subspace of $A^2V$. The radical of $\beta$ is therefore a proper $G_2(q)$-invariant subspace of $V$, which is irreducible, and hence $\beta$ is nondegenerate. However, for $V$ to admit a nondegenerate alternating form, it must have even dimension \cite[Proposition 2.4.1]{kleidman}. This is a contradiction, and thus $\dim(W) \ne 1$.

We therefore have from Lemma \ref{lem:g2irrmod} that $W$ has dimension seven if $p = 3$, and dimension seven or $14$ if $p > 3$. Since $\dim(U) = \dim(A^2V) - \dim(W)$, with $\dim(A^2V) = \binom{\dim(V)}{2} = 21$, the dimension of $U$ is $14$ if $p = 3$, and seven or $14$ if $p > 3$. Thus we are done if $p = 3$. For $p \ge 3$ in general, we know from Proposition \ref{prop:g2octirred} that $\scO$ is an irreducible $\mathbb{F}_q[G_2(q)]$-module of dimension seven. Hence $\scO \cong V^\alpha$ for some $\alpha \in \mathrm{Aut}(G_2(q))$ by Lemma \ref{lem:g2irrmod}, and this implies that $A^2\scO \cong A^2(V^\alpha)$. It is easy to see that $A^2(V^\alpha)$ and $(A^2V)^\alpha$ are equal as $\mathbb{F}_q[G_2(q)]$-modules, and that $G_2(q)$ stabilises the same subspaces of $A^2V$ and $(A^2V)^\alpha$, which are equal as vector spaces. Since $G_2(q)$ stabilises a $14$-dimensional subspace of $A^2\scO$ by Lemma \ref{lem:ftildecoll}\ref{lem:g214dim}, it follows that $\dim(U) = 14$.
\end{proof}

\section{Overgroups of $G_2(q)$ in $\mathrm{GL}(7,q)$}
\label{sec:overgroups}

In this section, we establish results about certain overgroups of $G_2(q)$ in $\mathrm{GL}(7,q)$ that we will use in \S\ref{sec:mainproof} to prove Theorem \ref{thm:main}. Although this theorem only applies to the case where $q$ is an odd prime, we consider here the general case with $q$ a power of an odd prime. We let $G:=G_2(q) < \mathrm{GL}(7,q)$ be the image of a fixed irreducible seven-dimensional $\mathbb{F}_q$-representation of $G_2(q)$. We know from Lemma \ref{lem:g2irrmod} that this representation is absolutely irreducible, and that all such representations are quasiequivalent. Furthermore, $G$ preserves a nondegenerate orthogonal form $\beta$. Let $\mathrm{GO}(7,q)$ (respectively, $\mathrm{CGO}(7,q)$) be the group of isometries (respectively, similarities) of $\beta$. In addition, let $\mathrm{SO}(7,q) := \mathrm{GO}(7,q) \cap \mathrm{SL}(7,q)$, and let $\Omega(7,q) := (\mathrm{SO}(7,q))'$. Bray \textit{et al.}~\cite[Table 8.40]{BHRD} show that $G$ is a maximal subgroup of $\Omega(7,q)$, and that $G$ lies in the Aschbacher class of $\mathrm{GO}(7,q)$ denoted by $\mathcal{S}$, consisting of the maximal subgroups that are absolutely irreducible, not of geometric type and that are, modulo a central subgroup, almost simple. Note that $G$ is perfect, as it is nonabelian and simple.

\begin{lem}
\label{lem:g2overcoll}
Let $Z:= Z(\mathrm{GL}(7,q))$ and $Z_0:= Z(\mathrm{SL}(7,q))$. Then:
\begin{enumerate}[label={(\roman*)}, font=\upshape]
\item the only maximal subgroup of $\mathrm{SL}(7,q)$ that contains $G$ is $Z_0 \mathrm{SO}(7,q)$; \label{lem:g2over}
\item if $G \le X \le \mathrm{SL}(7,q)$ and $\Omega(7,q) \not\le X$, then $X \le Z_0 G$; and \label{prop:g2inso7}
\item the group $Z G$ is the normaliser of $G$ in $\mathrm{GL}(7,q)$, and also the normaliser of $Z_0 G$ in $\mathrm{GL}(7,q)$. \label{lem:g27norm}
\end{enumerate}
\end{lem}

\begin{proof}
\begin{NoHyper}\ref{lem:g2over}\end{NoHyper} Bray \textit{et al.}~\cite[Tables 8.35--8.36]{BHRD} list the maximal subgroups of $\mathrm{SL}(7,q)$, as follows:
\begin{enumerate}[label={(\alph*)}]
\item groups that are images of reducible representations; \label{gp1}
\item groups $H$ whose perfect core $H^\infty$ is the image of a representation that is not absolutely irreducible; \label{gp2}
\item groups $H$ such that a conjugate of $H^\infty$ in $\mathrm{GL}(7,q)$ is defined over a subfield of $\mathbb{F}_q$ of prime index; \label{gp3}
\item groups isomorphic to $Z_0 \times \mathrm{PSU}(3,3)$, when $q$ is prime and $q \equiv 1 \mod 4$; \label{gp5}
\item groups isomorphic to $7_{+}^{1+2} \rtimes \mathrm{Sp}(2,7)$, when either $q$ is prime and $q \equiv 1 \mod 7$, or $q$ is a cube of a prime and $q \equiv 2, 4 \mod 7$; \label{gp6}
\item groups isomorphic to $(q-1)^6 \rtimes S_7$, when $q \ge 5$; \label{gp7}
\item groups isomorphic to $(q_0-1,7) \times \mathrm{SU}(7,q_0)$ that preserve a unitary form up to scalars, when $q = q_0^2$; and \label{gp4}
\item groups isomorphic to $Z_0 \mathrm{SO}(7,q)$ that preserve an orthogonal form up to scalars. \label{gp8}
\end{enumerate}
Since $G$ is an $\mathcal{S}$ subgroup of $\mathrm{GO}(7,q)$, no conjugate of $G$ in $\mathrm{GL}(d,q)$ is defined over a proper subfield of $\mathbb{F}_q$ \cite[Definition 2.1.3]{BHRD}. Additionally, $G$ is the image of an absolutely irreducible representation. As $G$ is perfect, all of these properties must hold for each group $H$ containing $G$, and for the perfect core $H^\infty$ of $H$. Hence $G$ does not lie in any of the groups in \begin{NoHyper}\ref{gp1}--\ref{gp3}\end{NoHyper}. We also see by considering orders that none of the groups in \begin{NoHyper}\ref{gp5}--\ref{gp7}\end{NoHyper} can contain $G$. Furthermore, by \cite[Lemma 1.8.8]{BHRD}, $G$ does not preserve any unitary form up to scalars, and does not preserve any bilinear form up to scalars other than the scalar multiples of $\beta$. Thus $G$ does not lie in any group in \begin{NoHyper}\ref{gp4}\end{NoHyper}. The groups in \begin{NoHyper}\ref{gp8}\end{NoHyper} preserve distinct orthogonal forms up to scalars, and it follows that the only maximal subgroup of $\mathrm{SL}(7,q)$ that can contain $G$ is $Z_0 \mathrm{SO}(7,q)$. This group does indeed contain $\Omega(7,q)$ and therefore $G$. Note that $|Z_0|=(q-1,7)$, and thus $Z_0 \mathrm{SO}(7,q) = \mathrm{SO}(7,q)$ when $q \not\equiv 1 \mod 7$.

\begin{NoHyper}\ref{lem:g27norm}\end{NoHyper} 
The outer automorphisms of $G$ consist of the field automorphisms, as well as the exceptional graph automorphism when $q$ is a power of $3$. None of these automorphisms leaves any irreducible seven-dimensional $\mathbb{F}_q[G]$-module invariant, and so by Clifford theory, these automorphisms are not realised in $\mathrm{GL}(7,q)$. Hence $N_{\mathrm{GL}(7,q)}(G) = C_{\mathrm{GL}(7,q)}(G)G$, which is equal to $ZG$ by Schur's lemma.
As $G = (ZG)' = (Z_0 G)'$, we also have $N_{\mathrm{GL}(7,q)}(G) = N_{\mathrm{GL}(7,q)}(Z_0 G)$.

\begin{NoHyper}\ref{prop:g2inso7}\end{NoHyper} Suppose that $X \le \mathrm{SL}(7,q)$ is such that $G \le X$ and $\Omega(7,q) \not\le X$. By part \ref{lem:g2over}, $X \le Z_0 \mathrm{SO}(7,q)$, and so $X \cap \Omega(7,q)$ is normal in $X$. Since $G$ is maximal in $\Omega(7,q)$, we have $G = X \cap \Omega(7,q)$. Thus $$X \le N_{\mathrm{GL}(7,q)}(G) \cap (Z_0 \mathrm{SO}(7,q)) = (Z G) \cap (Z_0 \mathrm{SO}(7,q)) \le (Z G) \cap \mathrm{SL}(7,q),$$ where we have used part \ref{lem:g27norm}. Applying Dedekind's identity gives $X \le Z_0 G$.
\end{proof}

Figure \ref{fig:g2q} summarises what we have proved about the overgroups of $G$ in $\mathrm{GL}(7,q)$. Note that $\Omega(7,q)$ is a maximal subgroup of $\mathrm{SO}(7,q)$, with index two \cite[Table 1.3]{BHRD}.

\clearpage

\begin{figure}[t]
\centering
\includegraphics{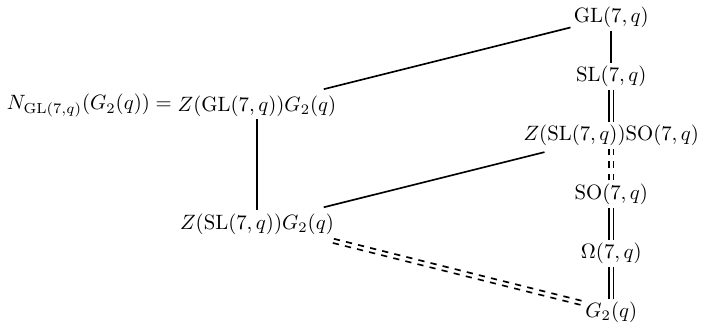}
\caption{Overgroups of $G_2(q)$ in $\mathrm{GL}(7,q)$, for an odd prime power $q$. Double edges indicate maximal containment, and subgroups connected by dashed edges are equal when $q \not\equiv 1 \mod 7$.}
\label{fig:g2q}
\end{figure}

\section{Proof of the main theorem}
\label{sec:mainproof}

In this section, we use the results of \S\ref{sec:reducibility}--\ref{sec:overgroups} to prove Theorem \ref{thm:main}. We retain the notation established at the start of \S\ref{sec:overgroups}, but with $q$ equal to an odd prime $p$. In addition, let $V$ be the module that affords an absolutely irreducible seven-dimensional representation with image $G = G_2(p) \le \mathrm{GL}(7,p)$, and let $Z:= Z(\mathrm{GL}(7,p))$ (note that this representation is unique when $p>3$). By Lemma \ref{lem:g2irrmod} and Proposition \ref{prop:g2octirred}, $V$ can in fact be any faithful seven-dimensional $\mathbb{F}_p[G_2(p)]$-module.

\begin{proof}[Proof of Theorem \begin{NoHyper}\ref{thm:main}\end{NoHyper}]
By Theorem \ref{thm:g2maxsubmod}, $G$ stabilises a proper $14$-dimensional subspace $U$ of $A^2V$, but no proper subspace of a higher dimension. Since the centre of the special linear group $\mathrm{SL}(7,p)$ acts via scalar multiplication, $J:= Z(\mathrm{SL(7,p)}) G$ also stabilises $U$. However, $\Omega(7,p)$ acts irreducibly on $A^2V$ \cite[Table 1]{liebeck85}. Thus no subgroup of $\mathrm{SL}(7,p)$ that contains $\Omega(7,p)$ stabilises $U$, and hence $J$ is the stabiliser of $U$ in $\mathrm{SL}(7,p)$ by Lemma \ref{lem:g2overcoll}\ref{prop:g2inso7}.

Now, let $H$ be the stabiliser of $U$ in $\mathrm{GL}(7,p)$. Then $J = H \cap \mathrm{SL}(7,p)$, and since $\mathrm{SL}(7,p)$ is normal in $\mathrm{GL}(7,p)$, it follows that $J$ is normal in $H$. Thus $H$ is contained in the normaliser of $J$ in $\mathrm{GL}(7,p)$, which is $ZG$ by Lemma \ref{lem:g2overcoll}\ref{lem:g27norm}. This normaliser certainly stabilises $U$, as $Z$ acts via scalar multiplication, and hence $H$ is this normaliser. As in \eqref{setofP}, set $P:=P_U$. We have therefore that the $p$-group $P$ has nilpotency class two, rank seven and exponent $p$, with $A(P) =ZG$, and with $|P| = p^{7+21-14} = p^{14}$. As we have chosen the dimension of $U$ to be as high as possible, this is in fact the smallest order of such a $p$-group.

Finally, if $P_1$ and $P_2$ are isomorphic $p$-groups, then the action of $A(P_1)$ on $P_1/\Phi(P_1) \cong V$ is equivalent to that of $A(P_2)$ on $P_2/\Phi(P_2) \cong V$. However, we know from Lemma \ref{lem:g2irrmod} that, when $p = 3$, there are two distinct isomorphism classes of irreducible seven-dimensional $G$-modules, corresponding to two inequivalent actions of $ZG$ on $V$. Therefore, when $p = 3$, there exist two nonisomorphic $p$-groups with the specified properties.
\end{proof}

It is easy to see that $ZG$ is not a maximal subgroup of $\mathrm{GL}(7,p)$, and hence Theorem \ref{thm:main} is not a consequence of the theory of Bamberg \textit{et al.}~\cite{bamberg}. Note also that when $p = 3$, the distinct $p$-groups of Theorem \ref{thm:main} correspond to distinct $14$-dimensional subspaces $U$ of $A^2V$. Computations in Magma \cite{magma} show, in the case of each irreducible module $V$, that $A^2V$ contains a unique $14$-dimensional submodule $U$, with $A^2V/U \cong V$ as $G$-modules. Lemma \ref{lem:g2irrmod} implies that this isomorphism also holds when $p > 3$.

\section{Concluding remarks}
\label{sec:concl}

We have constructed a $p$-group $P$, for each odd prime $p$, whose automorphism group induces $Z(\mathrm{GL}(7,p)) G_2(p)$ on $P/\Phi(P)$. There are several generalisations of this result that would be of interest. For each $p$-group $Q$ of nilpotency class two, rank seven and exponent $p$, the group $A(Q)$ contains the scalars of $\mathrm{GL}(7,p)$. However, we know from the work of Bryant and Kov\'acs \cite{bryant} that there exists a $p$-group $R$ of rank seven such that $A(R) = G_2(p)$. Hence in order to induce precisely $G_2(p)$ on the Frattini quotient of such a $p$-group, we need to consider $p$-groups of a higher nilpotency class and/or exponent. Indeed, in the sequel paper mentioned in \S\ref{sec:intro}, the second author uses this approach to construct such a $p$-group $R$.

We would also like to induce a group related to $G_2(q)$ for each prime power $q$, including $q$ even. One possible approach here would be to study $G_2(q)$ as a subgroup of a general linear group defined over a field of prime order. However, determining the stabilisers in such a general linear group of the corresponding exterior square's $G_2(q)$-submodules is a difficult problem, and the theory of representing linear groups on nonabelian $2$-groups is not yet complete. Note also that when $q$ is even, the smallest dimension of a nontrivial irreducible $\mathbb{F}_q[G_2(q)]$-module is six, and there is no irreducible module of dimension seven \cite[Appendix A.49]{lubeck}. Further exploration of the reducibility of $G_2(\mathbb{F})$-modules for a general field $\mathbb{F}$ would also be an interesting avenue of research. Finally, the aforementioned sequel paper continues the programme of study that we have begun in this paper. Namely, this sequel explores each family of nontwisted exceptional groups via a similar, but more Lie-theoretic, approach.

\subsection*{Acknowledgements}
We are very grateful to the referee for a detailed reading of the paper and recommendations.

\bibliographystyle{plain}
\bibliography{G2refs}

\end{document}